\documentclass{amsproc}
\usepackage{eurosym}
\usepackage{amssymb}
\usepackage{amsfonts}
\usepackage[utf8]{inputenc}
\usepackage[OT4]{fontenc}

\theoremstyle{plain}

\newtheorem{corollary}{Corollary}

\newtheorem{definition}{Definition}
\newtheorem{example}{Example}

\newtheorem{lemma}{Lemma}

\newtheorem{theorem}{Theorem}
\numberwithin{equation}{section}

\DeclareMathOperator{\lk}{lk}
\newcommand{\field}[1]{\mathbb{#1}}
\newcommand{\N}{\field{N}}

\begin{document}
\title{$f$-vectors implying vertex decomposability}
\author{Micha\l\ Laso\'{n}}
\thanks{This publication is partly supported by the Polish National Science Centre grant no. 2011/03/N/ST1/02918.}
\address{Theoretical Computer Science Department, Faculty of Mathematics and
Computer Science, Jagiellonian University, S.Łojasiewicza 6, 30-348 Krak\'{o}w, Poland}
\address{Institute of Mathematics of the Polish Academy of Sciences, \'{S}%
niadeckich 8, 00-956 Warszawa, Poland}
\email{michalason@gmail.com}

\begin{abstract}
We prove that if a pure simplicial complex $\Delta$ of dimension $d$ with $n$ facets has the least possible number of $(d-1)$-dimensional 
faces among all complexes with $n$ faces of dimension $d$, then it is vertex decomposable. 
This answers a question of J. Herzog and T. Hibi. In fact we prove a generalization of their theorem using combinatorial methods. 
\end{abstract}

\maketitle

\section{Introduction}

We call a simplicial complex \emph{pure} if all its facets are of the same dimension.

\begin{definition} A pure simplicial complex $\Delta$ of dimension $d$ and $n$ facets is called \emph{extremal} if it has the least possible number of 
$(d-1)$-dimensional faces among all complexes with $n$ faces of dimension $d$. 
\end{definition}

In particular, for $d=0$ all zero dimensional complexes are extremal, since all of them have exactly one $(-1)$-dimensional face, namely the empty set.
\smallskip 

In this paper we generalize and prove by only combinatorial means the following theorem of Herzog and Hibi form 1999.

\begin{theorem}
\label{main}\emph{(\cite{hehi}, Theorem 2.3)} An extremal simplicial complex is Cohen-Macaulay over an arbitrary field.
\end{theorem}

Their proof is algebraic and uses results from \cite{arhehi} and \cite{eare}. In fact they asked for a combinatorial proof. 
We give it by proving that an extremal simplicial complex is vertex decomposable. It is well-known that vertex decomposable complexes are Cohen-Macaulay.
Our proof goes along the lines of the proof of Kruskal-Katona inequality.
We start with a presentation of some necessary preliminaries.

\subsection{Vertex decomposable and Cohen-Macaulay complexes}

For a simplex $\sigma $ in a simplicial complex $\Delta$, the simplicial
complex 
$$\{\tau \in \Delta :\tau \cap \sigma =\emptyset\text{ and }\tau\cup \sigma \in \Delta \}$$ 
is called a \emph{link} of $\sigma $ in $\Delta $, and denoted by $\lk_{\Delta}\sigma$. 
For a vertex $x$ of $\Delta$ by $\Delta\setminus x$ we mean the simplicial complex $\{\tau \in \Delta : x\notin\tau\}$.

\begin{definition}
A pure simplicial complex $\Delta$ is vertex decomposable if one of the following holds:
\begin{enumerate}
\item $\Delta$ is empty,
\item $\Delta$ is a single vertex,
\item for some vertex $x$ both $\lk_{\Delta }\{x\}$ and $\Delta\setminus x$ are pure and vertex decomposable.
\end{enumerate}
\end{definition}

\begin{definition}
For a simplicial complex $\Delta $ on the set of vertices $\{1,\dots ,n\}$
and a given field $\mathbb{K}$, the \emph{Stanley-Reisner ring} (the \emph{face ring}) is 
$\mathbb{K}[\Delta]:=\mathbb{K}[x_{1},\dots ,x_{n}]/I$, where $I$ is
generated by all square-free monomials $x_{i_{1}}\cdots x_{i_{l}}$ for which 
$\{i_{1},\dots ,i_{l}\}$ is not a face in $\Delta $.
\end{definition}

When we say that a simplicial complex is \emph{Cohen-Macaulay} we always
mean that its Stanley-Reisner ring has this property.

The following is a folklore result (we refer the reader to e.g. \cite{bj}).

\begin{theorem}\label{bjorner}
For a simplicial complex $\Delta$ the following implications hold:
$$\Delta\text{ is vertex decomposable}\Rightarrow\Delta\text{ is shellable}\Rightarrow\Delta\text{ is  Cohen-Macaulay over any field}.$$
\end{theorem}

We recall a combinatorial description of Cohen-Macaulay complexes.

\begin{theorem}\label{cm}
\label{reisner} \emph{(\cite{re})} Let $R=\mathbb{K}[\Delta ]$ be
the face ring of $\Delta $. Then the following conditions are equivalent:
\begin{enumerate}
\item $R$ is Cohen-Macaulay ring,
\item $\tilde{H}_{i}(\lk_{\Delta}\sigma;\mathbb{K})=0$ for $i<dim(\lk_{\Delta }\sigma)$
for all simplices $\sigma \in \Delta $.
\end{enumerate}
\end{theorem}

For classical techniques of counting homologies we refer the reader to \cite{ha}, \cite{st}. For entertaining ones we advise Section 3.2 of \cite{lami}.

\subsection{Kruskal-Katona theorem}

One of the most natural questions concerning simplicial complexes is: \newline
\emph{What is the minimum number of $(k-1)$-element faces in simplicial complex with $n$ faces of
size $k$?}\newline 
This question was answered independently by Kruskal \cite{kr} and
Katona \cite{ka} in 1960's. For a positive integer $k$, they enlisted all $k$-element subsets of integers in the following order, called the 
\emph{squashed order}: $A<B$ if $\max (A\setminus B)<\max (B\setminus A)$. Let $%
S_{k}(n)$ be the set of first $n$ sets in this list. For a given set $\mathfrak{U}$
of $k$-element sets, denote by $\Delta \mathfrak{U}$ the set of all $%
(k-1)$-element sets which are contained in some member of $\mathfrak{U}$. The
Kruskal-Katona theorem reads as follows.

\begin{theorem}
\label{KrKt} For a positive integers $n,k$ and a set $\mathfrak{U}$ of $n$ sets of
size $k$ we have
\begin{equation*}
\left\vert {\Delta\mathfrak{U}}\right\vert \geq \left\vert {\Delta S_{k}(n)}%
\right\vert.
\end{equation*}
\end{theorem}

This result was further generalized by Clements and Lindstr\"{o}m in \cite{clli}. Daykin \cite{da1,da2} gave two simple proofs, 
and later Hilton \cite{hi}
gave another one. For an algebraic proof we refer the reader to \cite{arhehi}. We will work mainly with Hilton's idea. 

Note that the cardinality of $\Delta
S_{k}(n)$ may be easily determined. For a given $k$, each positive integer $n$ can be uniquely expressed as 
\begin{equation*}
n=\displaystyle\binom{a_{k}}{k}+\displaystyle\binom{a_{k-1}}{k-1}+\dots +%
\displaystyle\binom{a_{t}}{t},
\end{equation*}%
with $1\leq t\leq a_{t}$ and $a_{t}<\dots <a_{k}$. We have 
\begin{equation*}
\delta _{k-1}(n):=\left\vert {\Delta S_{k}(n)}\right\vert =\displaystyle%
\binom{a_{k}}{k-1}+\displaystyle\binom{a_{k-1}}{k-2}+\dots +\displaystyle%
\binom{a_{t}}{t-1}.
\end{equation*}

As a consequence of Kruskal-Katona theorem we get:

\begin{corollary}\label{ext} A pure simplicial complex $\Delta$ of dimension $d>0$ with $f$-vector $(f_{0},\dots,f_{d})$
is extremal if and only if $f_{d-1}=\delta_{d}(f_{d})$. 
\end{corollary}

\section{The main result}

For a better understanding of the assumption that $\Delta $ is extremal we will use Hilton's idea from his proof \cite{hi} of the Kruskal-Katona theorem.
First we define sets similar to $S_{k}(n)$. Let $S_{k}^{i}(n)$ denote the
first $n$ sets of $k$-element subsets of integers in the squashed
order ($A<B$ if $\max (A\setminus B)<\max (B\setminus A)$) which do not contain $i$. We also denote by $\{i\}(\cup)\mathfrak{U}$ 
the set $\{\{i\}\cup A:A\in \mathfrak{U}\}$. 

Let $\mathfrak{U}$ be a $n$-element set of $k$-element sets, let $V=\bigcup_{A\in\mathfrak{U}}A$ be an underlying set, 
and let $v$ be its cardinality. For $i\in V$, let $%
B_{i}=\{A\in\mathfrak{U}:i\notin A\}$, $C_{i}=\{A\setminus \{i\}:i\in A\in\mathfrak{U}\}$, and
let $b_{i},c_{i}$ be the respective cardinalities. Note that $c_{i}\neq 0$%
. We want to find an index $i$ such that $\left\vert {\Delta B_{i}}\right\vert
>\left\vert {C_{i}}\right\vert $.

\begin{lemma}
\label{i} Either there exists an $i$ such that $\left\vert {\Delta B_{i}}%
\right\vert >\left\vert {C_{i}}\right\vert $, or $\mathfrak{U}$ consists of all
possible $k$-element subsets of $V$.
\end{lemma}

\begin{proof}
We are going to count the sum of cardinalities of both sets when $i$ runs over all
elements of $V$. Then 
\begin{equation*}
\sum_{i\in V}\left\vert {\Delta B_{i}}\right\vert \geq kn(v-k)/(v-k)=kn=\sum_{i\in V}\left\vert {C_{i}}\right\vert, 
\end{equation*}%
since at left hand side each $A\in\mathfrak{U}$ gives $k$ distinct sets in its boundary, and it is counted once for every $i\notin A$. Some sets in boundaries of sets from $B_{i}$
can be the same, but their number is at most $(v-1)-(k-1)=v-k$. On the other
hand each $A\in\mathfrak{U}$ is counted $k$ times at the right side. Hence we can find a desired $i$, or
the above bounds are tight. In the latter case, when $A\in \Delta B_{i}$%
, all $v-k$ possibilities of completing it to a $k$-element set has to be in 
$\mathfrak{U}$. This means that $\mathfrak{U}$ consists of all possible $k$-element subsets of $V$ because from any set in $\mathfrak{U}$ we 
can delete any element and insert any other.
\end{proof}

\begin{lemma}
\label{characterization} If $\Delta $ is an extremal simplicial complex of positive dimension, then there exists a vertex $x$ such that both $\lk_{\Delta }\{x\}$ and $\Delta\setminus x$ are extremal.
\end{lemma}

\begin{proof}
Let $\Delta$ be of dimension $d-1>0$ and let $\mathfrak{U}$ be the set of all $d$-element sets in $\Delta$. If $\mathfrak{U}$ consists of all possible $d$-element subsets of a given $v$%
-element set, then the assertion of the lemma is clearly true (we can take any vertex). Otherwise,
due to Lemma \ref{i}, there exists an $i\in V$ such that $\left\vert {%
\Delta B_{i}}\right\vert >\left\vert {C_{i}}\right\vert $. We have that 
\begin{equation*}
\Delta\mathfrak{U}=\Delta B_{i}\cup C_{i}\cup (\{i\}(\cup )\Delta C_{i}).
\end{equation*}%
Since $\Delta B_{i}$ and $\{i\}(\cup )\Delta C_{i}$ are disjoint, it follows
that 
\begin{equation*}
\left\vert \Delta\mathfrak{U}\right\vert \geq \left\vert \Delta B_{i}\right\vert
+\left\vert \{i\}(\cup )\Delta C_{i}\right\vert > \left\vert {C_{i}}\right\vert
+\left\vert \{i\}(\cup )\Delta C_{i}\right\vert.
\end{equation*}%
So, by Theorem \ref{KrKt}, 
\begin{equation}
\left\vert \Delta\mathfrak{U}\right\vert \geq \left\vert \Delta
S_{d}^{i}(b_{i})\right\vert +\left\vert \{i\}(\cup )\Delta
S_{d-1}^{i}(c_{i})\right\vert ,  \label{1}
\end{equation}%
and
\begin{equation}
\left\vert \Delta\mathfrak{U}\right\vert >\left\vert S_{d-1}^{i}(c_{i})\right\vert
+\left\vert \{i\}(\cup )\Delta S_{d-1}^{i}(c_{i})\right\vert .  \label{2}
\end{equation}%
Since $\Delta S_{d}^{i}(b_{i})=S_{d-1}^{i}(e)$ for some $e$, there are now two possibilities:

\begin{enumerate}
\item $\Delta S_{d}^{i}(b_{i})\subset S_{d-1}^{i}(c_{i})$, then by (\ref{2})
we get\newline
$\left\vert \Delta\mathfrak{U}\right\vert >\left\vert S_{d-1}^{i}(c_{i})\right\vert
+\left\vert \{i\}(\cup )\Delta S_{d-1}^{i}(c_{i})\right\vert =\left\vert
\Delta (S_{d}^{i}(b_{i})\cup (\{i\}(\cup )S_{d-1}^{i}(c_{i})))\right\vert$,
which contradicts the assumption that $\Delta$ is extremal, since a complex generated by sets $S_{d}^{i}(b_{i})\cup (\{i\}(\cup )S_{d-1}^{i}(c_{i}))$ has  $b_i+c_i=\left\vert\mathfrak{U}\right\vert$ maximal faces.

\item $\Delta S_{d}^{i}(b_{i})\supset S_{d-1}^{i}(c_{i})$,
then by (\ref{1}) we obtain\newline
$\left\vert \Delta\mathfrak{U}\right\vert \geq \left\vert \Delta
S_{d}^{i}(b_{i})\right\vert +\left\vert \{i\}(\cup )\Delta
S_{d-1}^{i}(c_{i})\right\vert =\left\vert \Delta (S_{d}^{i}(b_{i})\cup
(\{i\}(\cup )S_{d-1}^{i}(c_{i})))\right\vert$,
and equality holds if and only if $C_{i}\subset \Delta B_{i}$,$\;\left\vert
\Delta B_{i}\right\vert =\left\vert \Delta S_{d}^{i}(b_{i})\right\vert $, and\newline
$\;\left\vert \Delta C_{i}\right\vert =\left\vert \{i\}(\cup )\Delta
S_{d-1}^{i}(c_{i})\right\vert =\left\vert \Delta
S_{d-1}^{i}(c_{i})\right\vert $.
\end{enumerate}

The complex $\Delta $ is extremal, so
equality holds, and we get that $C_{i}\subset \Delta B_{i}$ and $[B_i],[C_i]$ are extremal, where $[A]$ means the simplicial complex 
generated by the set of faces $A$. Observe that $\lk_{\Delta }\{i\}=[C_i]$ and $\Delta\setminus i=[B_i]$. The first equality is obvious, 
while the second is not  as clear. If $\sigma=\{v_1,\dots,v_k\}$ is a face in $\Delta\setminus i$ then it is a subface of some facet 
$F=\{v_1,\dots,v_d\}$. If $i$ does not belong to $F$ then $F\in [B_i]$ and so $\sigma$ does. Otherwise, $F\setminus \{i\}\in C_i\subset \Delta B_i$, 
so $F\setminus \{i\}\cup \{j\}\in B_i$ for some $j$. Hence $\sigma\in [B_i]$. Now $i=x$ gives the assertion.
\end{proof}

Finally, we are ready to prove the generalization of Theorem \ref{main}.

\begin{theorem}\label{mainn}
\label{main2} An extremal simplicial complex is vertex decomposable.
\end{theorem}
\begin{proof}
The proof goes by an induction on $d$ the dimension of $\Delta$ and secondly on the number of facets. 
If $d=0$ then $\Delta$ consists of points and by the definition it is vertex decomposable. When $d>0$, then by Lemma \ref{characterization} 
there exists a vertex $x$, such that both complexes $\lk_{\Delta }\{x\}$ and $\Delta\setminus x$ are extremal. The first is of lower dimension; 
and the second either has the same dimension as $\Delta$ but fewer facets, or it has smaller dimension. By the inductive hypothesis, 
both $\lk_{\Delta }\{x\}$ and $\Delta\setminus x$ are vertex decomposable, and as a consequence $\Delta$ also is.
\end{proof}

The above result is best possible in the following sense.  Let $\Delta$ be a pure simplicial complex of dimension $d>0$ with $f$-vector 
$(f_{0},\dots,f_{d})$, and with $f_{d-1}=\delta_{d}(f_{d})+c$, $c\in\N$. Due to Corollary \ref{ext} the meaning of Theorem \ref{mainn} is that if $c=0$, 
then $\Delta$ is vertex decomposable. But even for $c=1$ complex $\Delta$ does not have to be Cohen-Macaulay, which by Theorem \ref{bjorner} 
is a weaker property then vertex decomposability. We show the following.

\begin{example}
We have that $\delta_{d}(2)=2d+1$. Let $\Delta$ be a pure simplicial complex of dimension $d$ with the set of  facets $\mathfrak{U}$ consisting 
of two disjoint ones. Then $\left\vert\Delta\mathfrak{U}\right\vert=2d+2$, and $\tilde{H}_{0}(\lk_{\Delta}\emptyset;\mathbb{K})=1$, 
so due to Theorem \ref{cm} complex $\Delta$ is not Cohen-Macaulay over any field $\mathbb{K}$, and as a consequence it is also not vertex decomposable.
\end{example}

\section*{Acknowledgements}

I would like to thank Ralf Fr\"{o}berg and J\"{u}rgen Herzog for many inspiring conversations and introduction to this subject. 
I would also like to thank Jarek Grytczuk and Piotr Micek for help in preparation of this manuscript.

\end{document}